\documentclass[10pt,reqno]{article}

\usepackage[b5paper,top=1.2in,left=0.9in]{geometry}
\usepackage{verbatim}
\usepackage{amssymb}
\usepackage{amsmath}
\usepackage{graphicx}
\usepackage{appendix}
\usepackage{color}
\usepackage{amsthm}
\usepackage{tikz}
\usetikzlibrary{arrows,positioning,decorations.pathmorphing,  decorations.markings}

\renewcommand{\d}{\mathrm{d}}
\newcommand{\D}{\mathrm{D}}
\newcommand{\e}{\mathrm{e}}

\newtheorem{Thm}{Theorem}[section]
\newtheorem{Lem}[Thm]{Lemma}
\newtheorem{Prop}[Thm]{Proposition}
\newtheorem{Cor}[Thm]{Corollary}
\newtheorem{Rem}[Thm]{Remark}
\newtheorem{Def}[Thm]{Definition}
\newtheorem{Con}[Thm]{Conjecture}
\newtheorem{Ex}[Thm]{Example}

\newtheorem{Nota}[Thm]{Notation}

\newtheorem*{MainThm}{Main Theorem}

\newtheoremstyle{named}{}{}{\itshape}{}{\bfseries}{.}{.5em}{#1 #3}
\theoremstyle{named}

\def\R{\mathbb{R}}
\def\Q{\mathbb{Q}}

\def\C{\mathbb{C}}
\def\Z{\mathbb{Z}}

\def\fb{\mathfrak{b}}
\def\g{\mathfrak{g}}

\def\sl{\mathfrak{sl}}

\def\cH{\mathcal{H}}

\def\cK{\mathcal{K}}

\def\cN{\mathcal{N}}

\def\cP{\mathcal{P}}

\def\cR{\mathcal{R}}
\def\cS{\mathcal{S}}
\def\cT{\mathcal{T}}
\def\cU{\mathcal{U}}
\def\cV{\mathcal{V}}
\def\cW{\mathcal{W}}

\def\a{\alpha}
\def\b{\beta}
\def\c{\gamma}
\def\G{\Gamma}
\def\D{\Delta}

\def\d{\delta}
\def\e{\epsilon}
\def\ze{\zeta}

\def\k{\kappa}
\def\l{\lambda}

\def\S{\Sigma}

\def\t{\tau}

\def\W{\Omega}
\def\w{\omega}

\def\be{\textbf{e}}

\def\bf{\textbf{f}}

\def\bi{\textbf{i}}

\def\bo{\textbf{o}}

\def\bQ{\textbf{Q}}

\def\=>{\Longrightarrow}

\def\to{\longrightarrow}

\def\ox{\otimes}
\def\o+{\oplus}
\def\bo+{\bigoplus}
\def\x{\times}
\def\<{\langle}
\def\>{\rangle}
\def\({\left(}
\def\){\right)}
\def\oo{\infty}
\def\cong{\equiv}
\def\^{\wedge}
\def\+{\dagger}

\def\inv{^{-1}}
\def\half{\frac{1}{2}}

\def\dd[#1,#2]{\frac{d#1}{d#2}}
\def\del[#1,#2]{\frac{\partial #1}{\partial #2}}
\def\over[#1]{\overline{#1}}
\def\vec[#1]{\overrightarrow{#1}}

\def\tab{\;\;\;\;\;\;}

\newcommand{\til}[1]{\widetilde{#1}}

\newcommand{\binq}[2][cccccccccccccccccccccccccccccccccccccccccc]{\left[\begin{array}{#1}#2 \\ \end{array} \right]_q}
\newcommand{\binb}[1]{\begin{pmatrix}#1 \\ \end{pmatrix}_b}

\newcommand{\case}[2][cccccccccccccccccccccccccccccccccccccccccc]{\left\{\begin{array}{#1}#2 \\ \end{array}\right.}
\newcommand{\Eq}[1]{\begin{align}#1\end{align}}
\newcommand{\Eqn}[1]{\begin{align*}#1\end{align*}}

\begin{document}
\title{On tensor product decomposition of positive representations of $\cU_{q\til{q}}(\sl(2,\R))$}

\author{  Ivan C.H. Ip\footnote{
         	   Center for the Promotion of Interdisciplinary Education and Research/\newline
     	  Department of Mathematics, Graduate School of Science, Kyoto University, Japan
		\newline
		Email: ivan.ip@math.kyoto-u.ac.jp
          }
}

\date{\today}

\numberwithin{equation}{section}

\maketitle

\begin{abstract}
We study the tensor product decomposition of the split real quantum group $\cU_{q\til{q}}(\sl(2,\R))$ from the perspective of finite dimensional representation theory of compact quantum groups. It is known that the class of positive representations of $\cU_{q\til{q}}(\sl(2,\R))$ is closed under taking tensor product. In this paper, we show that one can derive the corresponding Hilbert space decomposition, given explicitly by quantum dilogarithm transformations, from the Clebsch-Gordan coefficients of the tensor product decomposition of finite dimensional representations of the compact quantum group $\cU_q(\sl_2)$ by solving certain functional equations and using normalization arising from tensor products of canonical basis. We propose a general strategy to deal with the tensor product decomposition for the higher rank split real quantum group $\cU_{q\til{q}}(\g_\R)$.
\end{abstract}

{\small  {\textbf{Keywords.} Modular double, quantum groups, tensor product, quantum dilogarithm}


\newpage
\tableofcontents
\section{Introduction}\label{sec:intro}
The notion of the \emph{positive principal series representations}, or simply \emph{positive representations}, was introduced in \cite{FI} as a new research program devoted to the representation theory of split real quantum groups $\cU_{q\til{q}}(\g_\R)$. It uses the concept of modular double for quantum groups \cite{Fa1, Fa2}, and has been studied for $\cU_{q\til{q}}(\sl(2,\R))$ by Teschner \textit{et al.} \cite{BT, PT1, PT2}. Explicit construction of the positive representations $\cP_\l$ of  $\cU_{q\til{q}}(\g_\R)$ associated to a simple Lie algebra $\g$ has been obtained for the simply-laced case in \cite{Ip2} and non-simply-laced case in \cite{Ip3}, where the generators of the quantum groups are realized by unbounded positive essentially self-adjoint operators acting on certain Hilbert spaces $L^2(\R^N)$. 

One important open problem is the study of the tensor product decomposition of the positive representations $\cP_{\l_1}\ox \cP_{\l_2}$. It is believed that the positive representations are closed under taking tensor product, which, together with the existence of the universal $R$ operator \cite{Ip7}, lead to the construction of new classes of \emph{braided tensor category} and hence to further applications parallel to those from the representation theory of compact quantum groups, including topological quantum field theory in the sense of Reshetikhin-Turaev \cite{RT1,RT2}, quantum higher Teichm\"{u}ller theory \cite{FG1,FG2}, and Toda conformal theory \cite{FL, Wy} which generalizes the quantum Liouville theory corresponding to the case of $\cU_{q\til{q}}(\sl(2,\R))$ \cite{PT1}.

Recently in \cite{Ip4}, using the theory of multiplier Hopf *-algebra, we showed that by considering the positive representations of $\cU_{q\til{q}}(\g_\R)$ restricted to its Borel part $\cU_{q\til{q}}(\fb_\R)$, the positive representations are actually closed under taking tensor product. This provides evidences that in general it may also be closed as well for the case of the full quantum group. The resulting intertwiner $\cT$ is called the quantum mutation operator in the theory of quantum higher Teichm\"{u}ller theory,  and it is a generalization of the earlier work by Frenkel-Kim \cite{FK} in the case of the modular double of the quantum plane, which is used to construct the quantum Teichm\"{u}ller space from the perspective of representation theory.

In the simplest case of $\cU_{q\til{q}}(\sl(2,\R))$, it is proved in \cite{PT2} that the positive representations are closed under taking tensor product in the sense of a continuous direct integral:
\Eq{
\cP_{\l_1}\ox \cP_{\l_2}\simeq \int_{\R_+}^{\o+}\cP_\a d\mu(\a)
}
for some Plancherel measure $d\mu(\a)$ (Theorem \ref{sl2trans}). The theorem is originally proved by considering the Casimir operator $\bQ$ acting on the tensor product and finding its spectral decomposition. The transformation is later refined in \cite{NT} and consists of composition of several unitary transformations intertwining the action of the generators of $\cU_{q\til{q}}(\sl(2,\R))$. Hence in another recent work \cite{Ip5}, we studied the generalized Casimir operators and the central characters of $\cU_{q\til{q}}(\g_\R)$ which are believed to be key ingredients to understand the tensor product decomposition in the higher rank case. In the process, we discovered the notion of \emph{virtual highest (resp. lowest) weight vectors}, which are certain generalized distributions such that the generators $E_i$ (resp. $F_i$) of $\cU_{q\til{q}}(\g_\R)$ act as zero. In particular, the expressions of the actions of the \emph{positive Casimirs} resembled the classical expressions given by the Weyl character formula. 

Another important observation is the analogy between the quantum factorials (or $q$-Gamma function $\G_q$ in general) and a remarkable special function called the (non-compact) \emph{quantum dilogarithm} $S_b$ \cite{FKa}, which is a generalization of the former to the ill-behaved case of $|q|=1$. In particular the non-compact version satisfies the same functional equation as the $q$-Gamma function (as well as the $q$-exponential function) up to some scaling. Using this observation, in \cite{Ip7} we construct the universal $R$ operator as a product of quantum dilogarithms, generalizing the compact expression which uses the $q$-exponential function, and shows in particular that \emph{braiding} exists for the positive representations. Recently in \cite{CDS}, starting from known solutions of the Yang-Baxter equation in the split real case, simple finite dimensional solutions are also constructed. 

In this paper, we bring together the above observations for the simplest case of $\cU_{q\til{q}}(\sl(2,\R))$. Let $V_N$ be the $N+1$ finite dimensional representation of the compact quantum group $\cU_q(\sl_2)$ spanned by $\{v_{N-2n}\}_{0\leq n\leq N}$. The tensor product decomposition $V_M\ox V_N \simeq \bo+_S V_S$ in terms of basis is then of the form
\Eq{\label{mainC}
v_{S-2k}=\sum C_{m,n,k}^{M,N,S} v_{M-2m}\ox v_{N-2n}
}
where $C_{m,n,k}^{M,N,S}$ (or more precisely the inverse) are called the Clebsch-Gordan coefficients. Their expressions are well-known in the literatures (see e.g. \cite{KS}). The main result of the paper is then the following (Theorem \ref{main} and Theorem \ref{main2})

\begin{MainThm}The integral transformation obtained from the compact case \eqref{mainC} by 
\begin{itemize}
\item[(1)] rescaling the Clebsch-Gordan coefficients according to the theory of canonical (crystal) basis \cite{Kashi,Lu2};
\item[(2)] replacing the $q$-factorials of the Clebsch-Gordan coefficients by the quantum dilogarithm functions $S_b$;
\item[(3)] replacing the integral parameters by real parameters using the expression from the virtual highest weights;
\item[(4)] replacing summation with appropriate contour integrals over the real line;
\end{itemize}
is a unitary transformation giving the tensor product decomposition of positive representations of the split real quantum group $\cU_{q\til{q}}(\sl(2,\R))$. The transformation coincides with the composition of the quantum mutation operator $\cT$ and the spectral decomposition of the Casimir operator up to a unitary scalar.
\end{MainThm}

We therefore propose the following strategy in the higher rank case. Find the finite dimensional Clebsch-Gordan coefficients, rescaled using the normalization from the theory of canonical basis, in terms of linear combinations of $q$-binomials, and replace them with the quantum dilogarithm functions and the summation range with the appropriate integrations. It turns out that the virtual highest weight again provides a correspondence between the finite and infinite expression. In particular the functional equation remains the same and should give an intertwiner between the tensor product of the positive representations. Finally one has to show that this integral transformation is well-defined and unitary in order to complete the proof. We will elaborate on these in the final section of the paper. In particular, together with the known quantum mutation operator $\cT$, we can obtain the spectral decomposition of the \emph{positive Casimirs} in the higher rank. We believe that even establishing the conjecture for lower rank case of $\cU_{q\til{q}}(\sl(3,\R))$ is enough to provide major breakthroughs in the theory of positive representations of split real quantum groups and its many applications as a completely new class of braided tensor categories.

The paper is organized as follows. In Section \ref{sec:prelim} we fix several notations and recall the properties of the quantum dilogarithm functions needed in this paper. In Section \ref{sec:findim}, we review the finite dimensional representation theory of $\cU_q(\sl_2)$ and reconstruct the Clebsch-Gordan coefficients by solving certain functional equation. In Section \ref{sec:posres}, we recall the definition of the positive representations of the split real quantum group $\cU_{q\til{q}}(\sl(2,\R))$ and calculate the intertwining transformation using the replacement rule (1)-(4) from the Main Theorem. In Section \ref{sec:tensor}, we construct a unitary transformation of the tensor product decomposition and show that the two transformations coincide. Finally in Section \ref{sec:higher} we give some remarks on the construction in the higher rank case and provide some examples in the special case of $\cU_{q\til{q}}(\sl(3,\R))$.

\section*{Acknowledgments}
This work is supported by the Top Global University Project, MEXT, Japan.
\section{Preliminaries}\label{sec:prelim}
In this section we fix the notation by recalling the definition of the compact quantum group $\cU_q(\sl_2)$ and its split real version $\cU_{q\til{q}}(\sl(2,\R))$, called the modular double, first introduced in \cite{Fa2}. Then we introduce the special function called the quantum dilogarithms studied extensively in \cite{FKa}, which is a non-compact analogue of the quantum factorials. We will recall its properties and its related integral transformations that are needed throughout the paper. Finally we fix the notion of certain unitary transformations on a Hilbert space that will be needed in this paper.
\subsection{Definition of $\cU_q(\sl_2)$ and $\cU_{q\til{q}}(\sl(2,\R))$}\label{subsec:def}
Let $q\in \C$ which is not a root of unity. In this paper, we will denote by $\cU_q(\sl_2)$ the Hopf algebra generated by $E,F$ and $K^{\pm1}$ subject to the following relations:
\Eq{
KE&=q^2 EK,\\
KF&=q^{-2}FK,\\
[E,F]&=\frac{K-K\inv}{q-q\inv},\\
KK\inv&=K\inv K=1.
}
We will choose the Hopf algebra structure of $\cU_q(\sl_2)$ given by
\Eq{\label{D(E)}
\D(E)&=E\ox 1+K\ox E,\\
\D(F)&=1\ox F+F\ox K\inv,\label{D(F)}\\
\D(K)&=K\ox K,\label{D(K)}\\
\e(E)&=\e(F)=0,\tab \e(K)=1,\\
S(E)&=-qE,\tab S(F)=-q\inv F,\tab S(K)=K\inv.
}

(However we will not need the counit $\e$ and antipode $S$ in this paper.)

In the split real case, we require $|q|=1$. Let $q=e^{\pi \bi b^2}$ where $b^2\in (0,1)\setminus\Q$. Then we define $\cU_q(\sl(2,\R))$ to be the real form of $\cU_q(\sl_2)$ induced by the star structure
\Eq{
E^*=E,\tab F^*=F,\tab K^*=K.
}
Finally, the modular double $\cU_{q\til{q}}(\sl(2,\R))$ is defined to be
\Eq{\label{modulardouble}
\cU_{q\til{q}}(\sl(2,\R)):=\cU_q(\sl(2,\R))\ox \cU_{\til{q}}(\sl(2,\R)),
}
where $\til{q}=e^{\pi \bi b^{-2}}$.
\subsection{Quantum factorials and quantum dilogarithms}\label{subsec:qd}
\begin{Def}
Let $q\in\C$ which is not a root of unity. We will denote the $q$-number by 
\Eq{[z]_q := \frac{q^z-q^{-z}}{q-q\inv},\tab z\in\C,}
the $q$-factorial by
\Eq{[n]!:=\prod_{k=1}^n [k]_q,\tab n\in\Z_{\geq 0}}
with $[0]!:=1$, and the $q$-binomial by
\Eq{\binq{n\\k}:= \case{\frac{[n]!}{[k]![n-k]!}& 0\leq k\leq n\in\Z,\\ 0& otherwise.}}
\end{Def}
The $q$-factorial $[n]!$ can be represented by a meromorphic function called the $q$-Gamma function $\G_q$ where
\Eq{\G_q(n+1)=[n]!,}
and satisfies the functional equation
\Eq{\G_q(z+1)=[z]_q \G_q(z)}
for $z\in\C$ when $|q|\neq 1$. However, in the split real setting, we need to consider the case when $|q|=1$, whence the definition is ill-defined. Instead, we will consider the following remarkable special function called the quantum dilogarithm discovered in \cite{FKa} and its variants.

In the following, let $q=e^{\pi i b^2}$ where $b^2\in (0,1)\setminus\Q$ and let $Q=b+b\inv$.

\begin{Def} The quantum dilogarithm function is defined for $0< Re(z)<Q$ as
\Eq{
G_b(z):= \over[\zeta_b]\exp\left(-\int_{C} \frac{e^{\pi t z}}{(e^{\pi b t}-1)(e^{\pi b\inv t}-1)}\frac{dt}{t}\right),
}
where $\zeta_b = e^{\frac{\pi i }{2}(\frac{b^2+b^{-2}}{6}+\half)}$ and the contour of integration goes over the poles of the denominator. It can be analytic continued to the whole complex plane as a meromorphic function with simple zeros at $z=Q+nb+mb\inv$ and simple poles at $z=-nb-mb\inv$ for $n,m\in\Z_{\geq 0}$.

We will also consider the following variants:
\Eq{
S_b(x)&:=e^{\frac{\pi i x(Q-x)}{2}}G_b(x),\\
g_b(x)&:=\frac{\over[\zeta_b]}{G_b(\frac{Q}{2}+\frac{\log x}{2\pi i b})}.
}
\end{Def}

The quantum dilogarithms enjoyed the following properties \cite{Ip1, PT2}:
\begin{Prop} Self-duality:
\Eq{S_b(x)=S_{b\inv}(x),\tab G_b(x)=G_{b\inv}(x).}

Functional equations: \Eq{\label{funceq}S_b(x+b^{\pm1})=-i(e^{\pi i b^{\pm1} x}-e^{-\pi i b^{\pm1} x})S_b(x),\tab G_b(x+b^{\pm 1})=(1-e^{2\pi ib^{\pm 1}x})G_b(x).}

Reflection property:
\Eq{\label{reflection}S_b(x)S_b(Q-x)=1,\tab G_b(x)G_b(Q-x)=e^{\pi
ix(x-Q)}.}

Complex conjugation: \Eq{\overline{S_b(x)}=\frac{1}{S_b(Q-\over[x])},\tab \overline{G_b(x)}=\frac{1}{G_b(Q-\bar{x})}.}

In particular for $x\in\R$:
\Eq{\left|S_b(\frac{Q}{2}+ix)\right|=\left|G_b(\frac{Q}{2}+ix)\right|=|g_b(e^{2\pi b x})|=1\label{gb1}.}

Asymptotic properties:
\Eq{\label{asymp} G_b(x)\sim\left\{\begin{array}{cc}\bar{\ze_b}&Im(x)\to+\oo,\\\ze_b
e^{\pi ix(x-Q)}&Im(x)\to-\oo.\end{array}\right.}
\end{Prop}

For two formal variables $u,v$ such that $uv=q^2 vu$, we have the $q$-binomial formula:
\Eq{\label{qbi-n}
(u+v)^n = \sum_{k=0}^n q^{-k(n-k)}\binq{n\\k} u^{n-k}v^k .
}
In fact we have a generalization to the split real case as well:
\begin{Lem} \label{qbi}\cite[B.4]{BT} $q$-binomial theorem:
For positive self-adjoint variables $u,v$ with $uv=q^2vu$, we have:
\Eq{\label{qbi-it}
(u+v)^{it}=\int_{\W}q^{\t(t-\t)}\binb{it\\i\t} u^{it-i\t}v^{i\t}d\t ,
}
where the $q$-beta function (or $q$-binomial coefficient) is given by
\Eq{\binb{t\\\t}=\frac{S_b(Q+bt)}{S_b(Q+b\t)S_b(Q+bt-b\t)},}
and $\W$ is the contour along $\R$ that goes above the pole at $\t=0$
and below the pole at $\t=t$.
\end{Lem}

\begin{Rem}\label{replaceqd}
The functional equation \eqref{funceq} for $S_b(x)$ implies that 
\Eq{\label{Sb(Q+bx)}\frac{S_b(Q+bx)}{S_b(Q+b(x-1))}=i(q-q\inv)\left[x\right]_q.}
In other words, the function $S_b(Q+bx)$ satisfies the same functional equation as $\G_q(z+1)$ up to a factor. Therefore by comparing the $q$-binomial formula \eqref{qbi-n} and \eqref{qbi-it}, one of the main philosophy to convert from compact case to the split real case is to express every expression in terms of the $q$-binomial function, and replacing the corresponding factorials $[n]!$ by $S_b(Q+bn)$. In that way the dependence on this extra factor $i(q-q\inv)$ will vanish.
\end{Rem}

Next, we recall some integral transformations involving the quantum dilogarithms. 
\begin{Nota}\label{contour}Unless otherwise specified, all the integral contours of $\int_\R dt$ in this paper are chosen to run along $\R$ and goes above all the poles of $G_b(\a-it)$, $G_b(\b+it)\inv$ and below all the poles of $G_b(\c+it)$, $G_b(\d-it)\inv$ for arbitrary $\a,\b,\c,\d\in\C$. Note that shifting the contour by $t\to \w-t$ for some $\w\in\R$ does not change this rule.
\end{Nota}
\begin{Lem}\label{FT} \cite[(3.31), (3.32)]{BT} We have the following Fourier transformation formula:
\Eq{\int_{\R} e^{2\pi i t r}\frac{e^{-\pi i t^2}}{G_b(Q+i
t)}dt=\frac{\bar{\ze_b}}{G_b(\frac{Q}{2}-ir)}=g_b(e^{2\pi b r}), }
\Eq{\int_{\R} e^{2\pi i t r}\frac{e^{-\pi Qt}}{G_b(Q+i
t)}dt=\ze_b G_b(\frac{Q}{2}-ir)=\frac{1}{g_b(e^{2\pi br})}=g_b^*(e^{2\pi br}).}
\end{Lem}
\begin{Lem}\label{tau} \cite[Lem 15]{PT2} We have the Tau-Beta Theorem:
\Eq{\int_\R e^{-2\pi \t \b}
\frac{G_b(\a+i\t)}{G_b(Q+i\t)}d\t =\frac{G_b(\a)G_b(\b)}{G_b(\a+\b)},}
By the asymptotic properties of $G_b$, the integral converges for $Re(\b)>0, Re(\a+\b)<Q$.
\end{Lem}

\begin{Lem}\label{45}Rewriting the integral transform in \cite{Vo} in terms of $G_b$, we obtain the 4-5 relation given by:
\Eq{\int_\R d\t e^{-2\pi \c\t}\frac{G_b(\a+i\t)G_b(\b+i\t)}{G_b(\a+\b+\c+i\t)G_b(Q+i\t)}=\frac{G_b(\a)G_b(\b)G_b(\c)}{G_b(\a+\c)G_b(\b+\c)},}
By the asymptotic properties of $G_b$, the integral converges for $Re(\c)>0$.
We can also rewrite both sides using the reflection property \eqref{reflection} and rearrange to get
\Eq{
\int_\R d\t e^{-2\pi Q \t}\frac{G_b(\a+i\t)G_b(\b+i\t)}{G_b(\a+\b+\c+i\t)G_b(Q+i\t)}=\frac{e^{-2\pi\a\b}G_b(\a)G_b(\b)G_b(\c)}{G_b(\a+\c)G_b(\b+\c)}.
}
\end{Lem}
\subsection{Unitary transformations}\label{subsec:unitary}
Often we will deal with unitary transformations $\cT:\cH\to\cH$ on the Hilbert space $\cH=L^2(\R\x \R, dxdy)$ which intertwines the action of the position and momentum operators. For any operator $P$ on $\cH$, its action under the transformation is given by 
\Eq{
P\mapsto \cT \circ P \circ \cT\inv.
}
In particular, let $x$ and $p:=\frac{1}{2\pi i}\del[,x]$ be self-adjoint operators acting on $L^2(\R)$ such that $[p,x]=\frac{1}{2\pi i}$. Then we have 
\begin{Lem}\label{unit1}
Multiplication on $f(x,y)\in L^2(\R\x\R)$ intertwines the action as follows:
\Eq{
e^{\pm 2\pi i xy}&: p_x\mapsto p_x\mp y, \tab p_y\mapsto p_y\mp x,\\
e^{\pm \pi i x^2}&: p_x\mapsto p_x\mp x,\\
e^{\pm 2\pi i p_x y}&: x\mapsto x\pm y,\tab  p_y\mapsto p_y\mp p_x.
}
\end{Lem}
The more important transformations are given by the quantum dilogarithm. By \eqref{gb1}, the operator $g_b(u)$ is unitary whenever $u$ is a positive self-adjoint operator, defined through the Fourier transformation formula in Lemma \ref{FT}.
\begin{Lem}\label{unit2}\cite{BT} Let $u,v$ be positive self-adjoint operators with $uv=q^2 vu$. Then 
\Eq{
g_b(u)^*vg_b(u) &= q\inv uv+v,\\
g_b(v)ug_b(v)^* &= u+q\inv uv.
}
These imply
\Eq{
g_b(u)g_b(v)&=g_b(u+v),\\
g_b(v)g_b(u)&=g_b(u)g_b(q\inv ub)g_b(v),
}
which are often referred to as the quantum exponential and quantum pentagon relations respectively.
\end{Lem}
As an example let $u=e^{2\pi b x}$ and $v=e^{2\pi b p}$. Here $v$ acts as shifting by $-ib$. Then $g_b(e^{2\pi b x})$ (as multiplication operator) maps
$$e^{2\pi b (x+p)}+e^{2\pi b p}\mapsto e^{2\pi b p}$$
by conjugation.

\section{Finite dimensional representations of $\cU_q(\sl_2)$}\label{sec:findim}
The theory of finite dimensional representations of the compact quantum group $\cU_q(\sl_2)$ is well known in the literature. Nice reviews can be found in \cite{CP, HK, KS}. In this paper, we will consider the following finite dimensional irreducible representations of $\cU_q(\sl_2)$:

\begin{Def}
Denote by $V_N$ the following $N+1$ dimensional irreducible representation spanned by the basis $\{v_{N-2n}\}_{0\leq n\leq N}=\{v_{N},v_{N-2},..., v_{-N}\}$ with the following action:
\Eq{
\label{E}E\cdot v_{N-2n} &= [n]_qv_{N-2n+2},\\
\label{F}F\cdot v_{N-2n} &= [N-n]_qv_{N-2n-2},\\
\label{K}K\cdot v_{N-2n} &= q^{N-2n} v_{N-2n}.
}
In particular, $v_N$ is the highest weight vector while $v_{-N}$ is the lowest weight vector.
\end{Def}
\begin{Rem} There exists another series of representations $V_N^-$ where the action of $E$ and $K$ carry a minus sign \cite{KS}. These two series constitute all finite dimensional irreducible representations of $\cU_q(\sl_2)$ for generic $q$. However we will not consider $V_N^-$ in this paper. 
\end{Rem}

\subsection{Clebsch-Gordan coefficients}\label{subsec:CG}
Consider the tensor product $V_M\ox V_N$ where the action of $\cU_q(\sl_2)$ is given by the coproduct \eqref{D(E)}-\eqref{D(K)}. Then it is well known that it decomposes into irreducibles as follows:
\Eq{
V_M\ox V_N \simeq \bo+_{\substack{S=|M-N|\\S\cong M+N\mbox{ (mod 2)}}}^{M+N} V_S.
}

Let us rename the basis vector and denote the basis of $V_M\ox V_N$ by $\{x_{M-2m}\ox y_{N-2n}\}$ where $0\leq m\leq M$ and $0\leq n \leq N$. For each $S$ with $|M-N|\leq S\leq M+N$ and $S\cong M+N \mbox{ (mod 2)}$, we denote the basis of $V_S$ by $\{v_{S-2k}\}$ where $0\leq k\leq S$. We want to establish explicitly the following transformation:
\Eq{\label{CG}
v_{S-2k} = \sum_{m,n} C_{m,n,k}^{M,N,S} x_{M-2m}\ox y_{N-2n},
}
where the matrix coefficients $C_{m,n,k}^{M,N,S}$ (or more precisely the inverse) are known as the Clebsch-Gordan coefficients. Their expressions are well-known in the literatures (see e.g. \cite{KS}) for other expressions of the representations $V_N$, which also take into account certain invariant bilinear form on the representations. However, it is instructive for us to re-derive the coefficients for our representation \eqref{E}-\eqref{K} since the same strategy will be used in the split real case. We will do so by solving the functional equations they have to satisfy in order to be an intertwiner of the action of $\cU_q(\sl_2)$. Similar methods have been used widely in the literature (see e.g. \cite{GKK, Ma}).

Since we fixed $M,N$ and $S$, for notation convenience we will write $C_{m,n}^k:=C_{m,n,k}^{M,N,S}$ when no confusion can arise.

\begin{Lem} \label{CG-K}$C_{m,n}^k$ is of the form 
\Eq{C_m^k\cdot \d_{m+n,k+d},}
where $d=\frac{M+N-S}{2}\in \Z_{\geq 0}$ and $\d_{mn}$ is the Kronecker delta function.
\end{Lem}
\begin{proof}
Consider the action of $K$ on both sides of \eqref{CG}. On the left hand side it acts as multiplication by $q^{S-2k}$ while on the right hand side it acts as $\D(K)=K\ox K$, which is multiplication by $q^{M-2m}q^{N-2n}$. Hence we need $S-2k=M+N-2m-2n$, which means $m+n=k+d$ as required.
\end{proof}
Note that $C_{m,n}^k$ is not identically zero only when $0\leq m\leq M$ and $0\leq k+d-m \leq N$. Also the condition $|M-N|\leq S \leq M+N$ implies $0\leq d \leq min(M,N)$.

\begin{Lem} \label{CG-E} Assume $C_m^k$ is not identically zero. Then it is of the form
\Eq{\label{CG-Eeq}
C_m^k=\sum_{r=0}^k (-q)^{r+m-k}q^{(r+m-k)(M-m)}\binq{k\\r}\binq{r+d\\k-m+d}c_r,
}
where $c_r:=C_{0,r+d}^r$ are constants that depends only on $r$ and $M,N,S$, and identically zero if the indices are outside of range.
\end{Lem}
\begin{proof}
We consider the action of $E$ on both sides of \eqref{CG}. We have
\Eqn{
E\cdot v_{S-2k}&=[k]_qv_{S-2k+2}\\
&=[k]_q\sum_{m+n=k-1+d} C_{m,n}^{k-1} x_{M-2m}\ox y_{N-2n},\\
E\cdot RHS&=[m]_q\sum_{m+n=k+d} C_{m,n}^{k} x_{M-2m+2}\ox y_{N-2n}+q^{M-2m}[n]_q\sum_{m+n=k+d} C_{m,n}^{k} x_{M-2m}\ox y_{N-2n+2}\\
&=\sum_{m+n=k-1+d}\left([m+1]_qC_{m+1,n}^k+q^{M-2m}[n+1]_qC_{m,n+1}^k\right)x_{M-2m}\ox y_{N-2n}.
}
Hence we need to solve the functional equation
\Eq{\label{Func-E}
[k]_qC_{m,n}^{k-1}=[m+1]_qC_{m+1,n}^k+q^{M-2m}[n+1]_qC_{m,n+1}^k,
}
or (shifting $m\to m-1$)
$$C_{m}^k=\frac{[k]_q}{[m]_q}C_{m-1}^{k-1}-q^{M-2m+2}\frac{[k+d-m+1]_q}{[m]_q}C_{m-1}^k.$$
Note that this functional equation is still valid when the indices are out or range.

Now by induction on $m$, we easily obtain
\Eqn{
C_{m}^k&=\sum_{r=0}^t (-1)^r q^{r(M-2m+t+1)}\frac{[m-t]!}{[m]!}\binq{t\\r}\frac{[k]!}{[k-t+r]!}\frac{[k+d-m+r]!}{[k+d-m]!}C_{m-t}^{k-t+r}
}
for $0\leq t\leq m$, which follows from the Pascal identity
\Eqn{
&q^{t-r}\binq{t\\r}+q^{-r-1}\binq{t\\r+1}=\binq{t+1\\r+1}.
}

Hence taking $t=m$,
\Eqn{
C_{m}^k &= \sum_{r=0}^m (-q)^{r}\frac{q^{r(M-m)}}{[m]!}\binq{m\\r}\frac{[k]!}{[k-m+r]!}\frac{[k+d-m+r]!}{[k+d-m]!}C_{0}^{k-m+r}.
}
Re-indexing the summation from $k-m+r\to r$ and simplifying, we obtain
\Eqn{
C_{m}^k = \sum_{r=k-m}^k (-q)^{r+m-k}q^{(r+m-k)(M-m)}\binq{k\\r}\binq{r+d\\k-m+d}C_{0}^{r}.
}
Finally note that $\binq{k\\r}=0$ for $r<0$ and $\binq{r+d\\k-m+d}=0$ for $r<k-m$. Hence we can rewrite the summation of $r$ as $\sum_{r=0}^k$.
\end{proof}

\begin{Lem} If $c_r:=C_{0,r+d}^r$ is not identically zero (i.e. $r+d\leq N$), then it is of the form:
\Eq{\label{cp2}
c_r = \binq{N-d\\r}\binq{S\\r}\inv c_0,
}
where $c_0:=C_{0,d}^0$ is a constant.
\end{Lem}
\begin{proof} If $c_r$ is not identically zero, then $c_{r'}$ are also not identically zero for $0\leq r'\leq r$. Now consider the action of $F$ on both sides of $\eqref{CG}$. Then we arrive at
\Eq{\label{Func-F}
[S-k]_q C_{m,n}^{k+1} = [N-n+1]_qC_{m,n-1}^k+q^{2n-N}[M-m+1]_qC_{m-1,n}^k.
}
When $m=0$, the last term on right hand side vanishes, hence we have the relation
\Eqn{
[S-k]_q c_{k+1} = [N-n+1]_qc_k=[N-k-d]_qc_k.
}
Hence we easily obtain
\Eq{\label{cp}
c_r &= \frac{[N-r-d+1]_q}{[S-r+1]_q}c_{r-1}\\
&=...=\binq{N-d\\r}\binq{S\\r}\inv c_0\nonumber
}
as required.
\end{proof}
Hence combining the above Lemmas, we obtain
\begin{Thm} \label{CG-full}The Clebsch-Gordan coefficients for the projection of $V_M\ox V_N$ onto $V_S$ is given by
\Eqn{
v_{S-2k} = \sum_{m+n=k+d} C_{m,n,k}^{M,N,S} x_{M-2m}\ox y_{N-2n},
}
where
\Eq{
C_{m,n,k}^{M,N,S}=\sum_{r=0}^k (-q)^{r+m-k}q^{(r+m-k)(M-m)}\binq{k\\r}\binq{r+d\\k-m+d}\binq{N-d\\r}\binq{S\\r}\inv c_0
}
for some constants $c_0:=c_0(M,N,S)$.
\end{Thm}
\begin{proof} We have shown directly that the coefficients $C_{m,n,k}^{M,N,S}$ intertwine the action of $E$ and $K$. By the finite dimensional representation theory, we just need to show that it provides the correct highest and lowest weight vector. Then the intertwining action of $F$ will be automatically satisfied.

For the highest weight vector, $k=0$. The condition that $|M-N|\leq S\leq M+N$ implies that the summation for $m+n=d$ ranges from $m=0$ to $m=d$. From the functional equation \eqref{Func-E}, the coefficients vanish identically after applying the action of $E$,
$$[m]_qC_{m,n-1}^k+q^{M-2m+2}[n]_qC_{m-1,n}^k=0,\tab m+n=d+1,$$
except for the initial case $m=0$ and the last case $n=0$, both of which has $[0]$ as the coefficient, hence the action vanishes.

Similarly for the lowest weight vector $k=S$, and the condition that $|M-N|\leq S\leq M+N$ implies the summation $m+n=S+d$ ranges from $m=M$ to $n=N$, hence the same argument on \eqref{Func-F} shows that the expression gives the correct lowest weight vector.
\end{proof}
It is known in the theory of tensor products of canonical (crystal) basis \cite{Kashi,Lu2} that, corresponding to the chosen coproduct \eqref{D(E)}, the highest weight vector of the component $V_S$ of $V_M\ox V_N$ can be normalized to the form
\Eq{\label{crystal}v_S = x_{M}\ox y_{N-2d}+ \sum_{m>0} P_m(q) x_{M-2m}\ox y_{N-2(d-m)}}
by setting $c_0:=C_{0,d}^0=1$, where $P_m(q):=C_{m,d-m}^0\in q\Z[q]$ is a polynomial in $q$ without leading constant. It turns out that this normalization is quite important as it implies that the split real version of the tensor product decomposition calculated in the next section is actually a unitary transformation.

\begin{Cor} Under the normalization \eqref{crystal}, the highest weight vector of the component $V_S$ of $V_M\ox V_N$ is given by
\Eq{
v_S &= \sum_{m=0}^d (-1)^{m}q^{m(M-m+1)}\binq{d\\m} x_{M-2m}\ox y_{N-2(d-m)}.}
\end{Cor}

\begin{Rem}
In the above calculations, we can as well choose the recurrence relation on $n$ instead of $m$, or we can even start with lowest weight vector and consider the action of $F$ instead. We then obtain different expressions representing the same coefficients (after normalization), equivalent to certain Pfaff-Saalsch\"{u}tz type $q$-binomial identities. However, it turns out that not all of them can be used in the argument presented in Section \ref{sec:tensor} when we generalize to the split real case, since certain integral transformations cannot be evaluated into closed form expressions.
\end{Rem}

\subsection{F-intertwiners}\label{subsec:Finter}
In the split real case, the representation is not finite dimensional and we do not have a highest weight vector. Hence the argument in the previous subsection might not be enough to show the intertwining action of $F$. However, intrinsically it should be equivalent to certain functional equations of the $q$-binomials. Hence let us show this directly.

\begin{Lem}\label{intertwine-F} The expression \eqref{CG-Eeq} together with the functional equation \eqref{cp} implies that $C_{m,n}^k$ intertwines the action of $F$.
\end{Lem}
\begin{proof}
We need to show \eqref{Func-F}. Starting from the right hand side, we have
\Eqn{
&[N-n+1]_q C_{m,n-1}^k+q^{2n-N}[M-m+1]_qC_{m-1,n}^k\\
=&[N-k-d+m]_q\sum_{r=0}^k (-q)^{r+m-k}q^{(r+m-k)(M-m)}\binq{k\\r}\binq{r+d\\k-m+d}c_r\\
&+q^{2(k+d-m+1)-N}[M-m+1]_q\sum_{r=0}^k (-q)^{r+m-k-1}q^{(r+m-k-1)(M-m+1)}\binq{k\\r}\binq{r+d\\k-m+d+1}c_r\\
=&\sum_{r=0}^k (-q)^{r+m-k}q^{(r+m-k)(M-m)}\binq{k\\r}\frac{[r+d]!}{[k-m+d+1]![r-k+m]!}\\
&\left([N-k-d+m]_q[k-m+d+1]_q-q^{k+r-S}[M-m+1]_q[r-k+m]_q\right)c_r.
}
Now using the identity (with $d=\frac{M+N-S}{2}$)
\Eq{
&[N-k-d+m]_q[k-m+d+1]_q-q^{k+r-S}[M-m+1]_q[r-k+m]_q\\
=&[N-d-r]_q[r+d+1]_q-q^{m-M-1}[S-k-r]_q[r-k+m]_q,
}
we have
\Eqn{
=&\sum_{r=0}^k (-q)^{r+m-k}q^{(r+m-k)(M-m)}\binq{k\\r}\frac{[r+d]!}{[k-m+d+1]![r-k+m]!}\\
&\left([N-d-r]_q[r+d+1]_q-q^{m-M-1}[S-k-r]_q[r-k+m]_q\right)c_r\\
=&\sum_{r=0}^k (-q)^{r+m-k}q^{(r+m-k)(M-m)}[N-d-r]_q\binq{k\\r}\binq{r+d+1\\k-m+d+1}c_r\\
&+\sum_{r=0}^k (-q)^{r+m-k-1}q^{(r+m-k-1)(M-m)}[S-k-r]_q\binq{k\\r}\binq{r+d\\k-m+d+1}c_r.
}
Now using \eqref{cp}, we rewrite $[N-d-r]c_r = [S-r]c_{r+1}$:
\Eqn{
=&\sum_{r=0}^k (-q)^{r+m-k}q^{(r+m-k)(M-m)}[S-r]_q\binq{k\\r}\binq{r+d+1\\k-m+d+1}c_{r+1}\\
&+\sum_{r=0}^k (-q)^{r+m-k-1}q^{(r+m-k-1)(M-m)}[S-k-r]_q\binq{k\\r}\binq{r+d\\k-m+d+1}c_r\\
=&\sum_{r=1}^{k+1} (-q)^{r+m-k-1}q^{(r+m-k-1)(M-m)}[S-r+1]_q\binq{k\\r-1}\binq{r+d\\k-m+d+1}c_{r}\\
&+\sum_{r=0}^k (-q)^{r+m-k-1}q^{(r+m-k-1)(M-m)}[S-k-r]_q\binq{k\\r}\binq{r+d\\k-m+d+1}c_r\\
=&\sum_{r=0}^{k+1} (-q)^{r+m-k-1}q^{(r+m-k-1)(M-m)}\binq{r+d\\k-m+d+1}\\
&\left([S-r+1]_q\binq{k\\r-1}+[S-k-r]_q\binq{k\\r}\right)c_r.
}
Finally using the generalized Pascal identity
\Eq{
[S-r+1]_q\binq{k\\r-1}+[S-k-r]_q\binq{k\\r}=[S-k]_q\binq{k+1\\r},
}
we arrive at
\Eqn{
&=[S-k]_q\sum_{r=0}^{k+1} (-q)^{r+m-k-1}q^{(r+m-k-1)(M-m)}\binq{k+1\\r}\binq{r+d\\k-m+d+1}c_r\\
&=[S-k]_qC_{m,n}^{k+1}
}
as required.
\end{proof}

\section{Positive representations $\cP_\l$ of $\cU_{q\til{q}}(\sl(2,\R))$}\label{sec:posres}
In \cite{FI, Ip2,Ip3}, a special class of irreducible representations for $\cU_{q\til{q}}(\g_\R)$, called the positive representations, is defined for the split real quantum groups corresponding to each simple Lie algebra $\g$. The generators of the quantum groups are realized by positive essentially self-adjoint operators acting on certain Hilbert space, and also satisfy the so-called \emph{transcendental relations}, relating the quantum group with its modular double counterpart. In this paper, we will just focus on the $\cU_{q\til{q}}(\sl(2,\R))$ case, which is studied extensively by Teschner \emph{et al.} from the work on quantum Liouville theory \cite{BT, NT, PT1, PT2}.
\subsection{Definitions}\label{subsec:def}
Let $x$ and $p=\frac{1}{2\pi i}\frac{d}{dx}$ be self-adjoint operators acting on $L^2(\R,dx)$ such that ${[p,x]=\frac{1}{2\pi i}}$. As before let $q=e^{\pi ib^2}$ where $b^2\in(0,1)\setminus\Q$.
\begin{Prop}\cite{Ip2}\label{canonicalsl2} Let $\l\geq 0$. The positive representation $\cP_\l$ of $\cU_{q\til{q}}(\sl(2,\R))$ acting on $L^2(\R$) is given by
\Eqn{
E&=\left(\frac{i}{q-q\inv}\right)(e^{\pi b(-x+\l-2p)}+e^{\pi b(x-\l-2p)})=\left[\frac{Q}{2b}+\frac{i}{b}(\l-x)\right]_qe^{-2\pi b p},\\
F&=\left(\frac{i}{q-q\inv}\right)(e^{\pi b(x+\l+2p)}+e^{\pi b(-x-\l+2p)})=\left[\frac{Q}{2b}+\frac{i}{b}(\l+x)\right]_qe^{2\pi b p},\\
K&=e^{-2\pi bx}.
}
Note that $\frac{i}{q-q\inv}\in \R_{>0}$, hence these are (unbounded) positive operators, and they are essentially self-adjoint.
\end{Prop}
\begin{Rem} This is Fourier transform of the representations given in \cite{BT,PT2} widely used in the physics literature. By Proposition \ref{real2finite} below, we see that our choice here is more natural in order to compare with the compact case.
\end{Rem}
\begin{Prop}
Let the rescaled generators be
\Eq{\be:=\left(\frac{i}{q-q\inv}\right)\inv E,\tab \bf:=\left(\frac{i}{q-q\inv}\right)\inv F.\label{smallef}}
Let the generators $\til{E}, \til{F}, \til{K}$ (as well as their rescaled version) be obtained by replacing $b$ with $b\inv$ in the representation of $E, F, K$. Then we have the transcendental relation
\Eq{\til{\be}:=\be^{\frac{1}{b^2}},\tab \til{\bf}:=\bf^{\frac{1}{b^2}},\tab \til{K}:=K^{\frac{1}{b^2}},\label{transdef}}
and the generators $E,F,K$ weakly commute with $\til{E},\til{F},\til{K}$, hence the above representation is consistent with the definition of the modular double \eqref{modulardouble}.
\end{Prop}

In \cite{Ip5}, the notion of a virtual highest (resp. lowest) weight vector is introduced. These are generalized functions such that formally the generators $E$ (resp. $F$) and also their modular double counterpart acts as zero on these vectors, and have been found for positive representations of $\cU_{q\til{q}}(\g_\R)$ for every simple type $\g_\R$. In particular, for $\cU_{q\til{q}}(\sl(2,\R))$ the virtual highest weight vector is given formally by $v=\d(x-\frac{iQ}{2}-\l)$. More precisely it means (note the conjugation)
\Eq{
(E\cdot f)(-\frac{iQ}{2}+\l)=(\til{E}\cdot f)(-\frac{iQ}{2}+\l)=0
}
 for any $f\in L^2(\R)$ in the common domain of $E$ and $\til{E}$, which consists of $L^2(\R)$ functions that admits analytic continuation to the strip $|Im(x)|< Q$, and hence well-defined.

Therefore, let us write formally
\Eq{\label{replace}
v_n&:= \d(x+\frac{ibn}{2}),\\
N&=-\frac{Q}{b}+\frac{2i\l}{b},\label{replace2}
}
such that $v_N =\d(x+\frac{ibN}{2})= \d(x-\frac{iQ}{2}-\l)$ is our highest weight vector. Then upon rewriting the action from Proposition \ref{canonicalsl2}, it becomes formally: (again they should be interpreted as actions on distributions)
\begin{Prop}\label{real2finite} Using the substitution rule \eqref{replace}-\eqref{replace2}, we have
\Eq{
E\cdot v_{N-2n} &= -[n]_qv_{N-2n+2},\\
F\cdot v_{N-2n} &= -[N-n]_qv_{N-2n-2},\\
K\cdot v_{N-2n} &= q^{N-2n}v_{N-2n}.
}
\end{Prop}
\begin{proof} For the action of $E$ we have
\Eqn{
E\cdot v_{N-2n} &= \left[\frac{Q}{2b}+\frac{i}{b}(\l-x)\right]_q\d(x+ib+\frac{ib(N-2n)}{2})\\
&= \left[\frac{Q}{2b}+\frac{i}{b}(\frac{Q+Nb}{2i}+\frac{ib(N-2n+2)}{2})\right]_q\d(x+\frac{ib(N-2n+2)}{2})\\
&= \left[\frac{Q}{b}+n-1\right]_qv_{N-2n+2}\\
&= -[n]_qv_{N-2n+2}
}
where we used $q^{\frac{1}{b^2}}=e^{\pi i} = -1$ and
\Eq{\label{minus}\left[\frac{1}{b^2}+z\right]_q = \frac{q^{\frac{1}{b^2}+z}-q^{-\frac{1}{b^2}-z}}{q-q\inv}=\frac{-q^z+q^{-z}}{q-q\inv}=-[z]_q.}
The calculations for $F$ and $K$ are similar.
\end{proof}

Hence formally as distribution, the positive representation \emph{resembles} the finite dimensional representation $V_N$ of the compact quantum group $\cU_q(\sl_2)$ up to a sign! (However obviously they are not equivalent since $\cP_\l$ is infinite dimensional.) This has an important implication, since it turns out the functional equation needed for the intertwiner in the tensor product decomposition does not depend on the sign. Hence by Remark \ref{replaceqd} we only need to replace the $q$-factorials in the Clebsch-Gordan coefficients by the corresponding quantum dilogarithms with the parameters replaced by the real version according to \eqref{replace}-\eqref{replace2}.

\subsection{Intertwiners of tensor product decompositions}\label{subsec:intertwiner}
Let us now consider the split real version of the Clebsch-Gordan coefficients constructed for the compact quantum group in Section \ref{subsec:CG}. Consider the tensor product of two positive representations $\cP_{\l_1}\ox \cP_{\l_2}$. Our goal is to construct a transformation of the form
\Eq{
\cP_{\l_1}\ox \cP_{\l_2}\to \int_{\R_+}^\o+ \cP_\a d\mu(\a),
}
where  $\int_{\R_+}^\o+ \cP_\a d\mu(\a)$ is a Hilbert space equivalent to $L^2(\R_+\x \R,d\mu(\a)dz)$ for some measure $d\mu(\a)$, and $\cU_{q\til{q}}(\sl(2,\R))$ acts on each slice $\cP_\a$ separately. Note that since $\cP_\a\simeq \cP_{-\a}$ (cf. \cite{Ip2, PT1}), the direct integral ranges only over $\a\in\R_+$. Let us ignore the measure $d\mu(\a)$ for the moment. As in the compact case, we will work with the inverse map, i.e. we want an integral transformation $C$ of the form
\Eq{\label{realtensor}
C:\int_{\R_+}^\o+ \cP_\a d\mu(\b)&\to \cP_{\l_1}\ox \cP_{\l_2},\nonumber\\
f(\a,z)&\mapsto F(x,y)\nonumber\\
&:=\iint C(x,y,\a,z)f(\a,z)d\a dz.
}

Since we need to work with contour integration, we will define the above integral transformation over entire functions $\cW_\a\widehat{\ox}\cW_z$ that decrease rapidly, where 
\Eq{
\cW_z=\{ e^{-r z^2+s z}Poly(z)|r\in \R_{>0}, s\in \C\}
}
is the \emph{core} of the domain of the positive unbounded self-adjoint operators $E,F,K$ \cite{Ip1} and we take the $L^2$-completion of the tensor product. We will show later that this is indeed a unitary transformation, in particular it is bounded, hence the integral transformation can be extended to the whole Hilbert space $L^2(\R_+\x \R)$.

\begin{Rem}
To compare with the compact case, formally in terms of the ``continuous basis" $\d(x-s)\d(y-t)$ for $L^2(\R\x\R,dxdy)$ and $\d(z-u)$ on the slice $\cP_\a\simeq L^2(\R,dz)$, this is
\Eq{
\d(z-u)&\mapsto\iint C(x,y,\a,z)\d(z-u)dz\nonumber\\
&=C(x,y,\a,u)\nonumber\\
&=\iint C(s,t,\a,u)\d(x-s)\d(y-t)dsdt.
}
\end{Rem}

\begin{Lem} $C(x,y,\a,z)$ is of the form $C(x,y,\a,z)\d(x+y-z)$.
\end{Lem}
\begin{proof} It follows from the intertwining action of $K$ on both sides of \eqref{realtensor}: $K$ acts on $f(\b,z)$ by $e^{-2\pi b z}$ while it acts on $F(x,y)$ by $e^{-2\pi b(x+y)}$.
\end{proof}
Now let's look at the action of $E$ and $F$. Let us assume that $C(x,y,\b,z)$ lies in the domain of $e^{\pm 2\pi b p}$ for $p=p_x,p_y,p_z$, i.e. it has no poles inside the strip $|Im(x)|\leq b, |Im(y)|\leq b, |Im(z)|\leq b$.
\begin{Lem}
$C(x,y,\a,z)$ satisfies the functional equations
\Eq{
\left[\frac{Q}{2b}-\frac{i}{b}(z-ib-\a)\right]_q C(x,y,\a,z-ib)=&\left[\frac{Q}{2b}-\frac{i}{b}(x-\l_1)\right]_qC(x+ib,t,\a,z)\nonumber\\
&+e^{-2\pi b x}\left[\frac{Q}{2b}-\frac{i}{b}(y-\l_2)\right]_qC(x,y+ib,\a,z),\label{eqE}\\
\left[\frac{Q}{2b}+\frac{i}{b}(z+ib+\a)\right]_q C(x,y,\a,z+ib)=&\left[\frac{Q}{2b}+\frac{i}{b}(y+\l_2)\right]_qC(x,y-ib,\a,z)\nonumber\\\label{eqF}
&+e^{2\pi b y}\left[\frac{Q}{2b}+\frac{i}{b}(x+\l_1)\right]_qC(x-ib,t,\a,z).
}
\end{Lem}
\begin{proof} It follows by considering the action of $E$ and $F$ on both sides of \eqref{realtensor}, i.e. by comparing the integral transformation $EC=CE$ and $FC=CF$. Since we assume $C(x,y,\b,z)$ lies in the domain of $e^{\pm 2\pi b p}$ for $p=p_x,p_y,p_z$, we can shift the contour of integration by $\pm ib$ and arrive at the functional equations.
\end{proof}

Now comes the important observation. We have seen that by formally defining the virtual highest weight vector, i.e. the substitution of \eqref{replace}-\eqref{replace2}, we recover the compact representations. Hence let us rewrite the above using (Note: they are not integers anymore)
\Eq{\label{replace3}
M:=-\frac{Q}{b}+\frac{2i\l_1}{b}, &&N&:=-\frac{Q}{2}+\frac{2i\l_2}{b}, &S&:=-\frac{Q}{b}+\frac{2i\a}{b},\nonumber\\
x=-\frac{ib(M-2m)}{2},&&y&:=-\frac{ib(N-2n)}{2}, &z&:=-\frac{ib(S-2k)}{2}.
}
Then the functional equations \eqref{eqE}-\eqref{eqF} become
\Eq{
[k]_qC'(m,n,S,k-1)&=[m+1]_qC'(m+1,n,S,k)+q^{M-2m}[n+1]_qC'(m,n+1,S,k),\\
[S-k]_qC'(m,n,S,k+1)&=[N-n+1]_qC'(m,n-1,S,k)+q^{2n-N}[M-m+1]_qC'(m-1,n,S,k),
}
where 
\Eq{
C'(m,n,S,k):=C\left(-\frac{ib(M-2m)}{2},-\frac{ib(N-2n)}{2},\frac{Sb+Q}{2i},-\frac{ib(S-2k)}{2}\right)
}
and we have used \eqref{minus}.

Now this is exactly the functional equations \eqref{Func-E} and \eqref{Func-F}! Therefore by Proposition \ref{CG-full} and Remark \ref{replaceqd}, we have the following
\begin{Thm}\label{main}
$C(x,y,\a,z)$ is given by
\Eq{\label{CG-real}
C(x,y,\a,z)=&\d(x+y-z)\frac{S_b(\frac{Q}{2}+i\a-iz)S_b(\frac{Q}{2}+i\a-i\l_1+i\l_2)}{S_b(\frac{Q}{2}-iy+i\l_2)S_b(2i\a)}\cdot \\
&\int_\R  \frac{e^{\pi(\frac{Q}{2}+ix+i\l_1)(r+\a-\l_1-y)}S_b(i\a-iy-i\l_1+ir)S_b(ir)S_b(2i\a+ir)}{S_b(\frac{Q}{2}+i\a-iz+ir)S_b(\frac{Q}{2}+i\a-i\l_1-i\l_2+ir)S_b(\frac{Q}{2}+i\a-i\l_1+i\l_2+ir)}dr,\nonumber
}
where the contour of integration follows our convention given in Notation \ref{contour}. This intertwines the action of $E,F$ and $K$ for the tensor product decomposition.
\end{Thm}
\begin{proof}
By Lemma \ref{CG-E}, and the substitution $[n]!\to S_b(Q+bn)$, for any $p\in\C$ any function of the form
\Eqn{
&(-q)^{r+m-k}q^{(r+m-k)(M-m)}\frac{S_b(Q+bk)}{S_b(Q+br)S_b(Q+bk-br)}\frac{S_b(Q+br+bd)}{S_b(Q+bk-bm+bd)S_b(Q+br-bk+bm)}
}
intertwines the action of $E$. Rewrite using the substitution rule we have
\Eqn{
=&\frac{e^{\pi(\frac{Q}{2}+ix+i\l_1)(r+\a-\l_1-y)}S_b(\frac{Q}{2}+i\a-iz)S_b(\frac{Q}{2}-i\a+i\l_1+i\l_2-ir)}{S_b(Q-ir)S_b(\frac{Q}{2}+i\a-iz+ir)S_b(\frac{Q}{2}-iy+i\l_2)S_b(Q+iy-i\a+i\l_1-ir)}\\
=&\frac{e^{\pi(\frac{Q}{2}+ix+i\l_1)(r+\a-\l_1-y)}S_b(\frac{Q}{2}+i\a-iz)S_b(i\a-iy-i\l_1+ir)S_b(ir)}{S_b(\frac{Q}{2}+i\a-iz+ir)S_b(\frac{Q}{2}-iy+i\l_2)S_b(\frac{Q}{2}+i\a-i\l_1-i\l_2+ir)}.
}
Here following the idea in the setting of multiplier Hopf algebra \cite{Ip7}, we have replaced $(-q)$ with $e^{\pi i b Q}$ and rescaled ${r\mapsto \frac{r}{ib}}$. Also we used $x+y=z$ and in the last line we used the reflection properties \eqref{reflection} of $S_b$.

Hence by Lemma \ref{intertwine-F} the general form for $C(x,y,\a,z)$ is given by
\Eq{\label{E-inter}
&\d(x+y-z)\int  \frac{e^{\pi(\frac{Q}{2}+ix+i\l_1)(r+\a-\l_1-y)}S_b(\frac{Q}{2}+i\a-iz)S_b(i\a-iy-i\l_1+ir)S_b(ir)}{S_b(\frac{Q}{2}+i\a-iz+ir)S_b(\frac{Q}{2}-iy+i\l_2)S_b(\frac{Q}{2}+i\a-i\l_1-i\l_2+ir)}c(r)dr
}
for some function $c(r)$ that satisfies the functional equation corresponding to \eqref{cp}:
\Eqn{
c(r)&=\frac{[N-\frac{r}{ib}-d+1]_q}{[S-\frac{r}{ib}+1]_q}c(r-ib)\\
&=\frac{\left[\frac{Q}{2b}+\frac{i}{b}(\a-\l_1+\l_2+r)\right]}{\left[\frac{i}{b}(r+2\a)\right]}c(r-ib).
}
Using the expression from \eqref{cp2}, similarly this can be solved by
\Eqn{
c(r) &= \frac{S_b(Q+bN-bd)}{S_b(Q+bN-bd-br)}\frac{S_b(Q+bS-br)}{S_b(Q+bS)}\\
&=\frac{S_b(\frac{Q}{2}+i\a-i\l_1+i\l_2)}{S_b(\frac{Q}{2}+i\a-i\l_1+i\l_2+ir)}\frac{S_b(2i\a+ir)}{S_b(2i\a)},
}
and hence we obtain the expression for $C(x,y,\a,z)$. Now we have to take care of the shifts in contour. From the analytic properties of $S_b$, we see that the expression allows the shift in $x, y$ and $z$ by $+ib$ only if $-\frac{Q}{2}< Im(r)<0$. On the other hand, we can shift $r\to r+x, r+y$ or $r+z$, and we see that it also allows the shift in all variables by $-ib$. Hence we will choose the contour of $r$ to go parallel along this strip and $C(x,y,\a,z)$ lies in the domain of $e^{\pm 2\pi b p}$ for $p=p_x,p_y,p_z$.

Finally according to the proof of Lemma \ref{intertwine-F}, in order for $C(x,y,\a,z)$ to intertwine the action of $F$, we need to be able to shift the contour of integration $\int dr$ by $ib$. Hence by the analytic properties of $S_b$, we will choose the contour of integration to go under the poles on the real line. By the asymptotic properties \eqref{asymp} the integral is absolutely convergent.
\end{proof}

The intertwiner $C(x,y,\a,z)$ above is obtained purely from the expression of the finite dimensional tensor product decomposition of the compact quantum group $\cU_q(\sl_2)$. In the next section, we will show that it equals (up to constant) a canonical unitary transformation which diagonalize the Casimir operator. In particular it shows that the integral transformation defined above is actually equivalent to a unitary transformation $C:\int_{\R_+}^\o+ \cP_\a d\mu(\a)\to \cP_{\l_1}\ox \cP_{\l_2}$.

\section{Tensor product decompositions of $\cP_{\l_1}\ox \cP_{\l_2}$}\label{sec:tensor}
In this section, we will study the tensor product decomposition by a series of unitary transformations given by quantum dilogarithms. An equivalent form of the transformation is given in \cite{NT}. Let us consider the rescaled operator $\be,\bf, K$ instead (cf. \eqref{smallef}), and define the Casimir operator by
\Eq{
Q&=FE+\frac{qK+q\inv K\inv}{(q-q\inv)^2},\\
\bQ&:=\left(\frac{i}{q-q\inv}\right)^{-2}Q = \bf\be-qK-q\inv K\inv,
}
such that $\bQ$ commutes with $E,F,$ and $K$ and acts on $\cP_\l$ as the scalar $e^{2\pi b\l}+e^{-2\pi b\l}$. 

\begin{Thm}\label{sl2trans}
The exists a unitary transformation 
\Eq{\Phi:\cP_{\l_1}\ox \cP_{\l_2} \to \int_{\R_+}^{\o+} \cP_\a d\mu(\a),} which can be decomposed into
\Eq{\label{Phidecomp}
\Phi = \left(\int_{\R_+}^\o+\cS_{\a}\inv d\mu(\a)\right) \circ (1\ox \cK)\circ (1\ox\cR)\circ \cT\circ (\cS_{\l_1}\ox \cS_{\l_2}),
}
and intertwines the action of $E,F$ and $K$. Here the intertwiners act as
\Eq{
\cS_{\l}: &\be=e^{\pi b(x-\l-2p_x)}+e^{\pi b(-x+\l-2p_x)}\to e^{-2\pi b p_x},\\
\cT:&\D(\be)\to 1\ox \be,\tab \D(K)\to 1\ox K,\\
\cR:&\D(\bQ)\to e^{2\pi b x}+e^{-2\pi b x}+e^{2\pi b p_x},
}
and $\cK$ is the Kashaev's transformation \cite{Ka} which gives the spectral decomposition of the ``Dehn twist" operator
\Eq{
\cK: e^{2\pi b x}+e^{-2\pi bx}+e^{2\pi b p_x} \to \int_{\R_+}^{\o+}(e^{2\pi b \a}+e^{-2\pi b \a}) d\mu(\a),
}
where the measure is given by 
\Eq{d\mu(\a) = |S_b(Q+2i\a)|^2d\a = 4\sinh(2\pi b \a)\sinh(2\pi b\inv \a)d\a.} Finally the last map means acting on each slice $\cP_\a$ by $\cS_\a\inv$.
\end{Thm}
\begin{proof} Let us treat in detail the effects of each transformation to the actions of the generators and their coproducts. Note that we can rewrite formally
\Eq{
\bf = (\bQ+qK+q\inv K\inv)\be\inv.
}
Hence we only need to keep track of the action of $K,\be$ and $\bQ$ only. We will be using repeatedly the unitary transformations given in Section \ref{subsec:unitary}. Also since the action is intertwined by conjugation, we will omit constants by writing $X\simeq Y$ if $X=cY$ for some constants $c\in\C$ with $|c|=1$.

First of all the \emph{simplification transformation} $\cS_\l$ on $L^2(\R,dx)$ is given by the multiplication operator
\Eqn{
\cS_\l &=S_b(\frac{Q}{2}-ix+i\l)\\
&\simeq e^{-\frac{\pi i}{2}x(x-2\l)}g_b(e^{2\pi b (-x+\l)})\\
&=(2p_x\mapsto 2p_x+x-\l)\circ g_b(e^{2\pi b (-x+\l)}),
}
which keeps $K, \bQ$ invariant and transform
\Eqn{
\be =&e^{\pi b(x-\l-2p_x)}+e^{\pi b(-x+\l-2p_x)}\to e^{-2\pi b p_x},\\
\bf =&e^{\pi b(x+\l+2p_x)}+e^{\pi b(-x-\l+2p_x)}\to e^{2\pi b(x+p_x)}+e^{2\pi b(-x+p_x)}+e^{2\pi b(\l+p_x)}+e^{2\pi b(-\l+p_x)}.
}
Next the \emph{Borel part decomposition operator} $\cT$ is studied in \cite{Ip4} for any split real quantum groups $\cU_{q\til{q}}(\g_\R)$. In the simplest case of $\cU_{q\til{q}}(\sl(2,\R))$ it is studied in \cite{FK} as the \emph{quantum mutation operator} between representations of the quantum plane, and is given by
\footnote{We used a different coproduct here but the transformation is equivalent up to multiplication by $e^{2\pi i xy}$.}
$$\cT = (y\to y-x)\circ e^{2\pi i x y}g_b^*(e^{2\pi b(x-p_x+p_y)}),$$
which transforms
\Eqn{
\D(\be) = \be\ox 1 +K\ox \be &= e^{-2\pi b p_x}+e^{-2\pi b x-2\pi bp_y}\\
&\to e^{-2\pi bp_y}=1\ox \be,\\
\D(K)=K\ox K &=e^{-2\pi bx-2\pi b y} \\
&\to e^{-2\pi b y}=1\ox K,
}
and
\Eqn{
\D(\bf)=&1 \ox \bf+\bf\ox K\inv \\
=&e^{2\pi b(y+p_y)}+e^{2\pi b(-y+p_y)}+e^{2\pi b(\l_2+p_y)}+e^{2\pi b(-\l_2+p_y)}\\
&+e^{2\pi b(y+x+p_x)}+e^{2\pi b(y-x+p_x)}+e^{2\pi b(y+\l_1+p_x)}+e^{2\pi b(y-\l_1+p_x)}\\
\to&e^{2\pi b(-y+p_y)}+e^{2\pi b(-x-p_x+p_y)}+e^{2\pi b(\l_2-x+p_y)}+e^{2\pi b(-\l_2-x+p_y)}\\
&+e^{2\pi b(x+p_x+p_y)}+e^{2\pi b(y+p_y)}+e^{2\pi b(-x+p_x+p_y)}+e^{2\pi b(\l_1+p_x+p_y)}+e^{2\pi b(-\l_1+p_x+p_y)},
}
so that $\D(\bQ)$ becomes
\Eqn{
&\D(\bQ)=\D(\bf)\D(\be)-q\D(K)-q\inv \D(K\inv)\\
&\to e^{2\pi b(\l_2-x)}+e^{2\pi b(-\l_2-x)}+e^{2\pi b(x+p_x)}+e^{2\pi b(-x-p_x)}+e^{2\pi b(-x+p_x)}+e^{2\pi b(\l_1+p_x)}+e^{2\pi b(-\l_1+p_x)}.
}
Note that $\bQ$ does not depend on the second variable, since it should commute with $K$ and $E$ which forms an irreducible representation of the Weyl pair. 

Now the \emph{Casimir simplification transformation} $\cR$ can be chosen to be
\Eqn{
\cR&=(x\to x-\l_2)\circ e^{2\pi i x\l_1}g_b^*(e^{2\pi b(-x-\l_1)})g_b(e^{2\pi b(x-\l_1)})\circ g_b^*(e^{2\pi b(\l_2-p_x)}).
}
Again this obviously keeps $\D(\be)$ and $\D(K)$ invariant since it only involves the variable $y$. Let us rename $y\to z$. One can then check using Lemma \ref{unit1} and \ref{unit2} repeatedly that it maps 
\Eqn{
\D(\bQ) &\to e^{2\pi b x}+e^{-2\pi b x}+e^{2\pi b p_x},\\
\D(\bf) &\to (e^{2\pi bx}+e^{-2\pi bx}+e^{2\pi bp_x})e^{2\pi b p_z}+e^{2\pi b(z+p_z)}+e^{2\pi b(-z+p_z)}.
}

Next, \emph{Kashaev's transformation} $\cK$ is given by
\Eq{
\cK:f(y)\mapsto F(\a):=\int_\R S_b(ix-i\a)S_b(ix+i\a)f(x)dx,
}
with inverse given by
\Eq{\label{Kashaevinv}
\cK\inv : F(\a)\mapsto f(x):=\int_0^\oo S_b(-ix+i\a)S_b(-ix-i\a)F(\a)d\mu(\a).
}
Here the contour of both integration follows Notation \ref{contour}. It is shown in \cite{Ka} that the functions
$\Psi_\a=S_b(ix-i\a)S_b(ix+i\a)$ with $\a>0$ form a complete set of orthogonal eigendistributions of the operator $\bQ$ with eigenvalues $e^{2\pi b\a}+e^{-2\pi b\a}$, hence the transformation intertwines $\D(\bQ)$ and it is a unitary transformation with the required measure $d\mu(\a)$.

Hence we now have
\Eqn{
\D(\be) &= e^{-2\pi b p_z},\\
\D(\bf) &= e^{2\pi b(z+p_z)}+e^{2\pi b(-z+p_z)}+e^{2\pi b(\a+p_z)}+e^{2\pi b(-\a+p_z)},\\
\D(K)&=e^{-2\pi b z}.
}
Finally applying $S_\a\inv$ on each slice $\cP_\a$ recovers the actions on $\int_{\R_+}^{\o+}\cP_\a d\mu(\a)$ as required.
\end{proof}
\begin{Rem} The simplification transformation $S_\l=S_b(\frac{Q}{2}-ix+i\l)$ is precisely the real analogue of the compact transformation given by the crystal basis \cite{Kashi, Lu1} of $\cU_q(\sl_2)$ mapping $E^n \to \frac{E^n}{[n]!}$ and hence the action becomes just shifting the index.
\end{Rem}
Now we would like to compare this unitary transformation with what we get from the compact quantum group decomposition, namely the integral transformation \eqref{realtensor} given by Theorem \ref{main}. 
\begin{Thm}\label{main2} The integral transformation 
\Eqn{
\int_{\R_+}^{\o+}\cP_\a d\mu(\a)&\to \cP_{\l_1}\ox \cP_{\l_2},\\
f(\a,z)&\mapsto F(x,y):=\int_{\R}\int_{\R_+} C(x,y,\a,z)f(\a,z)d\a dz,
}
with $C(x,y,\a,z)$ given by Theorem \ref{main} is the same as the unitary transformation $\Phi\inv$ up to a constant, where $\Phi$ is given by \eqref{Phidecomp}, and the constant depends only on $\l_1,\l_2$ and $\a$ which produces the Plancherel measure $d\mu(\a)$.
\end{Thm}
The rest of this section will be devoted to the proof of this Theorem.

First we start with the following observation:
\begin{Lem}\label{TinvRinv}
The transformations $\cT\inv$ and $\cR\inv$ are given as integral transformations by
\Eq{
\cT\inv\cdot f(x,y) &= \int_\R \frac{e^{2\pi i y(t-x)}}{G_b(Q+it)}f(x-t, x+y)dt,\\
\cR\inv\cdot f(x) &=\int\frac{e^{2\pi(\frac{Q}{2}+i\l_2)(x+\l_2-t)}G_b(it-ix-i\l_2)}{S_b(\frac{Q}{2}+it+i\l_1-i\l_2)S_b(\frac{Q}{2}+it-i\l_1-i\l_2)}f(t)dt.
}
\end{Lem}
\begin{proof}
Using Lemma \ref{FT}, we can write
\Eqn{
g_b(e^{2\pi b(x-p_x+p_y})&=\int_\R \frac{e^{2\pi itx-2\pi i t p_x+2\pi i t p_y}e^{-\pi i t^2}}{G_b(Q+it)}dt\\
&=\int_\R \frac{e^{2\pi itx-2\pi i t^2}}{G_b(Q+it)}e^{-2\pi i t p_x+2\pi i t p_y}dt,
}
where we have used $e^{2\pi i t x-2\pi i t p_x} = e^{-\pi i t^2}e^{2\pi i t x}e^{-2\pi i t p_x}$. Finally using $e^{-2\pi i t p_x+2\pi i t p_y}f(x,y) = f(x-t,y+t)$ we obtain the claim.

Similarly 
\Eqn{
g_b(e^{2\pi b(\l_2-p_x)})&=\int_\R \frac{e^{2\pi i t\l_2}e^{-2\pi i t p_x}e^{-\pi i t^2}}{G_b(Q+it)}dt=\int_\R e^{2\pi i t\l_2}e^{\pi Qt}G_b(-it)e^{-2\pi i t p_x}dt.
}
Combining with the definition of $\cR$, changing $g_b$ to $S_b$ and shifting the contour by $t\to x+\l_2-t$ gives the claim.
\end{proof}
From the proof of Theorem \ref{main}, we know that in order to intertwine the action of $E$ and $K$, the kernel $C(x,y,\a,z)$ is of the form
\Eq{C(x,y,\a,z)=&\d(x+y-z)\int_\R  e^{\pi(\frac{Q}{2}+ix+i\l_1)(r+\a-\l_1-y)}\nonumber\\
&\frac{S_b(\frac{Q}{2}+i\a-iz)}{S_b(\frac{Q}{2}-iy+i\l_2)}\frac{S_b(i\a-iy-i\l_1+ir)}{S_b(\frac{Q}{2}+i\a-iz+ir)}c(r,\a)dr\label{E-inter2}
}
for some function $c(r,\a)$. Here we have absorbed the factor $\frac{S_b(ip)}{S_b(\frac{Q}{2}+i\a-i\l_1-i\l_2+ip)}$ from \eqref{E-inter} into $c(r,\a)$.

Obviously the factor $S_b(\frac{Q}{2}+i\a-iz)$ is related to the simplification operator $S_{\a}$. Hence let us define
\Eq{
C'(x,y,\a,z):=S_b(\frac{Q}{2}+i\a-iz)\inv C(x,y,\a,z).
}
\begin{Lem}\label{Lem-cp}
If 
\Eq{
c(r,\a) =\frac{e^{-2\pi i \a(r+\a-\l_1)}}{G_b(\frac{Q}{2}-i\a+i\l_1)},
}
then the integral transformation
\Eq{\label{E-inter3}
f(x,y)\mapsto \int_{\R} C'(x,y,\a,x+y)f(\a,x+y)d\a
}
coincides with the unitary transformation $(\cS_{\l_1}\ox \cS_{\l_2})\inv \circ \cT\inv$.
\end{Lem}
\begin{proof}
By Lemma \ref{TinvRinv}, the map $(\cS_{\l_1}\ox \cS_{\l_2})\inv \circ \cT\inv $ is given by
\Eqn{
f(x,y)\mapsto S_b(\frac{Q}{2}-ix+i\l_1)\inv S_b(\frac{Q}{2}-iy+i\l_2)\inv\int_\R \frac{e^{-2\pi i t(x+t)}}{G_b(Q+it)}f(x+t,x+y)dt.
}
On the other hand, we have
\Eqn{
&\int_\R C'(x,y,\a,x+y)f(\a,x+y)d\a\\
=&\iint_{\R\x\R} \frac{e^{\pi(\frac{Q}{2}+ix+i\l_1)(r+\a-\l_1-y)}S_b(i\a-iy-i\l_1+ir)}{S_b(\frac{Q}{2}+i\a-ix-iy+ir)S_b(\frac{Q}{2}-iy+i\l_2)}c(r,\a)f(\a,x+y) dr d\a\\
=&\iint_{\R\x\R} \frac{e^{\pi(\frac{Q}{2}+ix+i\l_1)(r+\a-\l_1-y)}e^{-2\pi i\a(r+\a-\l_1)}S_b(i\a-iy-i\l_1+ir)}{S_b(\frac{Q}{2}+i\a-ix-iy+ir)S_b(\frac{Q}{2}-iy+i\l_2)G_b(\frac{Q}{2}-i\a+i\l_1)}f(\a,x+y) dr d\a,
}
where the contour of both integrals follow Notation \ref{contour}. Rewriting $S_b$ as $G_b$ and shifting the contour of $r$ we have
\Eqn{
=&\int_\R \frac{e^{2\pi i\a(x+\l_1-\a)}e^{-2\pi i x y}e^{\frac{\pi i}{2}(x+\l_1)^2}e^{-\frac{\pi i}{8}Q^2}}{S_b(\frac{Q}{2}-iy+i\l_2)G_b(\frac{Q}{2}-i\a+i\l_1)}\cdot\\
&\left(\int_\R\frac{e^{-2\pi r(i\a-ix)}G_b(i\a-iy-i\l_1+ir)}{G_b(\frac{Q}{2}+i\a-ix-iy+ir)}dr\right)f(\a,x+y)d\a\\
=&\int_\R \frac{e^{-2\pi i \a y-h(\frac{Q}{2}-i\a+i\l_1)+\frac{1}{2}h(\frac{Q}{2}+ix-i\l_1)+h(i\a-ix)}}{S_b(\frac{Q}{2}-iy+i\l_2)G_b(\frac{Q}{2}-i\a+i\l_1)}\cdot\\
&\left(\int_\R\frac{e^{-2\pi r(i\a-ix)}G_b(\frac{Q}{2}+ix-i\l_1+ir)}{G_b(Q+ir)}dr\right)f(\a,x+y)d\a,
}
where $h(x):=\pi ix(Q-x)$. By Lemma \ref{tau} it equals
\Eqn{
=&\int_\R \frac{e^{-2\pi i \a y-h(\frac{Q}{2}-i\a+i\l_1)+\frac{1}{2}h(\frac{Q}{2}+ix-i\l_1)+h(i\a-ix)}}{S_b(\frac{Q}{2}-iy+i\l_2)G_b(\frac{Q}{2}-i\a+i\l_1)} \frac{G_b(\frac{Q}{2}+ix-i\l_1)G_b(i\a-ix)}{G_b(\frac{Q}{2}+i\a-i\l_1)} f(\a,x+y)d\a\\
=&\frac{S_b(\frac{Q}{2}+ix-i\l_1)}{S_b(\frac{Q}{2}-iy+i\l_2)}\int_{\R} \frac{e^{-2\pi i \a y}}{G_b(Q+ix-i\a)}f(\a,x+y)d\a\\
=&S_b(\frac{Q}{2}-ix+i\l_1)\inv S_b(\frac{Q}{2}-iy+i\l_2)\inv\int_\R \frac{e^{2\pi i y (t-x)}}{G_b(Q+it)}f(x-t,x+y)dt\\
=&(\cS_{\l_1}\ox \cS_{\l_2})\inv \circ \cT\inv,
}
where in the last line we have renamed $\a\to x-t$.
\end{proof}
We would like to compute the transformation $\Phi\inv$ given by \eqref{Phidecomp}, and at the same time keeping the kernel of the form \eqref{E-inter2}. We will use the following trick:
\begin{Lem} If $\Psi$ and $D$ are unitary integral operators of the form
\Eq{
\Psi:f(\a)&\mapsto f(x):=\int_\W \Psi(\a,x)f(\a)d\a,\\
D:f(x)&\mapsto f(x):=\iint_{\R\x\R} \left(D(p+\a,x)c(p,\a)dp\right)f(\a)d\a,
}
for some contour $\W\subset\R$ of $\a$, then
\Eq{
D\circ \Psi: f(\a)\mapsto \int_\W\left(\int_\R D(p+\a,x)\int_\R c(p+\b,\a-\b)\Psi(\a,\a-\b)d\b dp\right)f(\a)d\a,
}
where the outside integral (and measure) $\int_\W d\a$ is kept invariant. In particular this means the middle factor is transformed as 
\Eq{
c(p,\a)\mapsto \int_\R c(p+\b,\a-\b)\Psi(\a,\a-\b)d\b.
}
\end{Lem}
\begin{proof}
$D\circ \Psi$ is given by
\Eqn{
D\circ \Psi: f(\a)\mapsto f(x):=\iint_{\R\x\R}\left(D(p+\b,x)c(p,\b)dp\right)\int_\W \Psi(\a,\b)f(\a)d\a d\b.
}
Since the integral operator is assumed to be unitary, the integrals are absolutely convergent and we can interchange the order of integration to get
\Eqn{
D\circ \Psi: f(\a)\mapsto f(x):=\int_\W\left(\iint_{\R\x\R}D(p+\b,x)c(p,\b)\Psi(\a,\b)d\b dp\right)f(\a)d\a.
}
Now we shift $p\to p-\b+\a$ and $\b\to \a-\b$ to obtain the result. 
\end{proof}
Now we see that our kernel $C'(x,y,\a,z)$ from \eqref{E-inter3} is precisely of the form as in the Lemma, hence we can compute $\Phi\inv$ by applying the Lemma twice on the function $c(p,\a)$ from Lemma \ref{Lem-cp}.

Using the expression of $\cR\inv$ from Lemma \ref{TinvRinv}, $(\cS_{\l_1}\ox \cS_{\l_2})\inv\circ \cT\inv \circ (\cR\ox 1)\inv$ equals \eqref{E-inter3} for 
\Eqn{
&c(p,\a)\\
=& \int_\R \frac{e^{-2\pi i(\a-\b)(r+\a-\l_1)}}{G_b(\frac{Q}{2}-i\a+i\b+i\l_1)}\frac{e^{2\pi(\frac{Q}{2}+i\l_2)(-\b+\l_2)}G_b(i\b-i\l_2)}{S_b(\frac{Q}{2}+i\a+i\l_1-i\l_2)S_b(\frac{Q}{2}+i\a-i\l_1-i\l_2)}d\b\\
=& \frac{e^{2\pi i\l_1(\l_1-r)+2\pi i\l_2(\l_2-\a)+\pi Q r+\frac{\pi i Q^2}{2}}G_b(\frac{Q}{2}-i\a-i\l_1+i\l_2)}{G_b(\frac{Q}{2}+i\a-i\l_1-i\l_2)}\\
&\int_\R \frac{e^{-2\pi \b(\frac{Q}{2}-i\a+i\l_1+i\l_2-ir)}G_b(\frac{Q}{2}+i\a-i\l_1-i\l_2+i\b)}{G_b(Q+i\b)}d\b.
}
Using Lemma \ref{tau}, we arrive at
\Eqn{
c(p,\a)&= \frac{e^{2\pi i\l_1(\l_1-r)+2\pi i\l_2(\l_2-\a)+\pi Q r+\frac{\pi i Q^2}{2}}G_b(\frac{Q}{2}-i\a-i\l_1+i\l_2)}{G_b(\frac{Q}{2}+i\a-i\l_1-i\l_2)}\cdot\\
&\frac{G_b(\frac{Q}{2}+i\a-i\l_1-i\l_2)G_b(\frac{Q}{2}-i\a+i\l_1+i\l_2-ir)}{G_b(Q-ir)}\\
&=\frac{e^{-\pi i (\a-\l_1)^2+\pi i(\l_1-\l_2)^2+\pi i\l_1^2+2\pi i r(\l_2-\a)+\frac{\pi i Q^2}{4}}G_b(\frac{Q}{2}-i\a-i\l_1+i\l_2)G_b(ir)}{G_b(\frac{Q}{2}+i\a-i\l_1-i\l_2+ir)}.
}

Finally, we apply Kashaev's inverse map $\cK\inv$ given by \eqref{Kashaevinv}. By Lemma \ref{Lem-cp}
the full map $\Phi\inv$ then equals \eqref{E-inter3} for 
\Eqn{
&c(p,\a)\\
=&e^{-\pi i (\a-\b-\l_1)^2+\pi i(\l_1-\l_2)^2+\pi i\l_1^2+2\pi i (r+\b)(\l_2-\a+\b)+\frac{\pi i Q^2}{4}}\\
&\frac{G_b(\frac{Q}{2}-i\a-i\l_1+i\l_2+i\b)G_b(ir+i\b)}{G_b(\frac{Q}{2}+i\a-i\l_1-i\l_2+ir)}S_b(-i\b)S_b(-i\b+2i\a)S_b(\frac{Q}{2}+i\a-iz)d\b\\
=&\frac{e^{-3\pi i\a^2+2\pi i\a\l_1+\pi i(\l_1-\l_2)^2+2\pi i r(\l_2-\a)+\pi Q\a+\frac{\pi i Q^2}{4}}S_b(\frac{Q}{2}+i\a-iz)}{G_b(\frac{Q}{2}+i\a-i\l_1-i\l_2+ir)}\\
&\int_\R e^{-2\pi \b(\frac{Q}{2}+i\l_1-i\l_2-i\a-ir)}\frac{G_b(\frac{Q}{2}-i\a-i\l_1+i\l_2+i\b)G_b(ir+i\b)}{G_b(Q+i\b)G_b(Q-2i\a+i\b)d\b}.
}

Using Lemma \ref{45} we obtain
\Eqn{
c(p,\a)=&\frac{e^{-3\pi i\a^2+2\pi i\a\l_1+\pi i(\l_1-\l_2)^2+2\pi i r(\l_2-\a)+\pi Q\a+\frac{\pi i Q^2}{4}}}{G_b(\frac{Q}{2}+i\a-i\l_1-i\l_2+ir)}\\
&\frac{G_b(\frac{Q}{2}-i\a-i\l_1+i\l_2)G_b(\frac{Q}{2}+i\l_1-i\l_2-i\a-ir)G_b(ir)}{G_b(Q-2i\a-ir)G_b(\frac{Q}{2}+i\l_1-i\l_2-i\a)}\\
=&\frac{e^{\pi i(\l_1^2+\l_2^2-\a^2)}e^{\frac{\pi i Q^2}{4}}S_b(\frac{Q}{2}+i\a-iz)S_b(\frac{Q}{2}+i\a-i\l_1+i\l_2)S_b(ir)S_b(2i\a+ir)}{S_b(\frac{Q}{2}+i\a+i\l_1-i\l_2)S_b(\frac{Q}{2}+i\a-i\l_1-i\l_2+ir)S_b(\frac{Q}{2}+i\a-i\l_1+i\l_2+ir)}
}

\begin{proof}[Proof of Theorem \ref{main2}]
Combining with \eqref{E-inter2}, we obtain
\Eqn{
\Phi\inv &= (\cS_{\l_1}\ox \cS_{\l_2})\inv\circ \cT\inv \circ (\cR\ox 1)\inv \circ (\cK\ox 1)\inv \circ (1\ox \cS_{\a})\\
f(\a,z)&\mapsto F(x,y):=\int_\R\int_0^\oo C(x,y,\a,z)f(\a,z) d\mu(\a) dz,
}
where
\Eq{
C(x,y,\a,z)&=\d(x+y-z)\int_\R  e^{\pi(\frac{Q}{2}+ix+i\l_1)(r+\a-\l_1-y)}e^{\pi i(\l_1^2+\l_2^2-\a^2)}e^{\frac{\pi i Q^2}{4}}\cdot\nonumber\\\label{CG-real2}
&\frac{S_b(\frac{Q}{2}+i\a-iz)}{S_b(\frac{Q}{2}-iy+i\l_2)}\frac{S_b(i\a-iy-i\l_1+ir)}{S_b(\frac{Q}{2}+i\a-iz+ir)}\cdot\\
&\frac{S_b(\frac{Q}{2}+i\a-i\l_1+i\l_2)S_b(ir)S_b(2i\a+ir)}{S_b(\frac{Q}{2}+i\a-i\l_1-i\l_2+ir)S_b(\frac{Q}{2}+i\a+i\l_1-i\l_2)S_b(\frac{Q}{2}+i\a-i\l_1+i\l_2+ir)}dr\nonumber.
}
Compare with \eqref{CG-real}, we see that
\Eq{
C(x,y,\a,z)_{\eqref{CG-real2}}&=const(\a,\l_1,\l_2)C(x,y,\a,z)_{\eqref{CG-real}},
}
where
\Eqn{
const(\a,\l_1,\l_2)&=\frac{e^{\pi i(\l_1^2+\l_2^2-\a^2)}e^{\frac{\pi i Q^2}{4}}S_b(2i\a)}{S_b(\frac{Q}{2}+i\a+i\l_1-i\l_2)}.
}
In particular $$|const(\a,\l_1,\l_2)|^2 =|S_b(2i\a)|^2=|S_b(Q+2i\a)|^{-2},$$
and hence the $L^2$ measure $d\mu(\a)$ of the transformation given by \eqref{CG-real} actually coincides with the one given by \eqref{CG-real2}.
\end{proof}

\begin{Rem} Using the substitution rule \eqref{replace}-\eqref{replace2}, some of the transformations presented in Theorem \ref{sl2trans} descend to transformations by classical $\G$ function, except $\cR$ and $\cK$ which diagonalize the Casimir operator. Therefore the same argument cannot be applied to the tensor product decomposition of \emph{classical} $SL(2,\R)$ representations. This is clearly due to the fact that the Casimir operator under tensor product of unitary representations of $SL(2,\R)$ consists of both continuous and discrete spectrum, which differs from the quantum case of $\cU_{q\til{q}}(\sl(2,\R))$.
\end{Rem}

\section{Remarks on higher rank}\label{sec:higher}
We have seen in Proposition \ref{real2finite} that using the notion of virtual highest weight vector, one can replace the variables of the positive representations $\cP_\l\simeq L^2(\R)$ of $\cU_{q\til{q}}(\sl(2,\R))$ and recover the finite dimensional representations (up to a sign) of the compact quantum group $\cU_q(\sl_2)$. This can obviously be generalized to higher rank split real quantum group $\cU_{q\til{q}}(\g_\R)$ since the virtual highest weight vectors are also known. 

In particular, the positive representation of $\cU_{q\til{q}}(\g_\R)$ is given by $\cP_{\vec[\l]}\simeq L^2(\R^{l(w_0)})$ and parametrized by $\vec[\l]\in\R_{\geq 0}^{rank(\g)}$, where $l(w_0)$ is the length of the longest element $w_0$ of the Weyl group. Following Proposition \ref{real2finite}, what we get will be an expression of the irreducible finite dimensional representation $V_{\vec[N]}$ of the corresponding compact quantum group $\cU_q(\g_c)$, parametrized by $rank(\g)$ positive integers $\vec[N]\in\Z_{\geq 0}^{rank(\g)}$, and represented by a finite dimensional $\C$-vector subspace of $l^2(\Z_{\geq0}^{l(w_0)})$.

\begin{Ex}
Let us consider the positive representations of $\cU_{q\til{q}}(\sl(3,\R))$ generated by $\{E_i,F_i,K_i\}_{i=1,2}$ on $L^2(\R^3,dudvdw)$ given by \cite{Ip2}:

\Eq{
E_1&=\left[\frac{Q}{2b}-\frac{i}{b}u\right]_q e^{2\pi b(-p_u-p_v+p_w)}+\left[\frac{Q}{2b}+\frac{i}{b}(v-w)\right]_q e^{-2\pi bp_v},\\
E_2&=\left[\frac{Q}{2b}-\frac{i}{b}w\right]_q e^{-2\pi bp_w},\\
F_1&=\left[\frac{Q}{2b}+\frac{i}{b}(-u+v+2\l_1)\right]_q e^{2\pi bp_v},\\
F_2&=\left[\frac{Q}{2b}+\frac{i}{b}(u+2\l_2)\right]_q e^{2\pi bp_u}+\left[\frac{Q}{2b}+\frac{i}{b}(2u-v+w+2\l_2)\right]_qe^{2\pi bp_w},\\
K_1&=e^{2\pi b(u-2v+w)},\\
K_2&=e^{2\pi b(-2u+v-2w)}.
}
Following the notion of virtual highest weight calculated in \cite{Ip5}, we make a shift $u\to u-\l_2,v\to v-\l_1-\l_2,w\to w-\l_1$, and using similar arguments of Proposition \ref{real2finite}, we write $$v_{k,m,n} = \d(u+\frac{ib(N_2-2k)}{2})\d(v+\frac{ib(N_1+N_2-2m)}{2})\d(w+\frac{ib(N_1-2n)}{2})$$
and obtain an expression for the finite dimensional representation $V_{N_1,N_2}$ of $\cU_q(\sl_3)$:
\Eq{
E_1\cdot v_{k,m,n}&= [k]_q v_{k-1,m-1,n+1}+[m-n]v_{k,m-1,n},\\
E_2\cdot v_{k,m,n}&= [n]_q v_{k,m,n-1},\\
F_1\cdot v_{k,m,n}& = [N_1+k-m]_qv_{k,m+1,n},\\
F_2\cdot v_{k,m,n}& = [N_2-k]_qv_{k+1,m,n}+[N_2-2k+m-n]_q v_{k,m,n+1},\\
K_1\cdot v_{k,m,n}& = q^{k-2m+n+N_1}v_{k,m,n},\\
K_2\cdot v_{k,m,n}& = q^{-2k+m-2n+N_2}v_{k,m,n},
}
where we have removed the minus signs arising from the conversion. This seems to be a new explicit expression for the representation $V_{N_1,N_2}$ that does not use a weight diagram. We see that the highest weight vector is given by $v_{0,0,0}$ and the lowest weight vector is given by $v_{N_2,N_1+N_2,N_1}$. In general the basis of the representation $V_{N_1,N_2}$ can be obtained for example by taking the action of the canonical basis \cite{Lu1}:
\Eqn{F_1^a F_2^b F_1^c\cdot v_{0,0,0},\tab a+c\leq b,c\leq N_1,\\
F_2^a F_2^b F_1^c \cdot v_{0,0,0},\tab a+c< b,c\leq N_2,
}
and it will be a vector subspace of the vector space $\cV$ spanned by $\{v_{k,m,n}\}$ with $0\leq k\leq N_2, 0\leq m\leq N_1+N_2, 0\leq n\leq N_1$.
\end{Ex}
Now the tensor product decomposition of two finite dimensional representations of $\cU_q(\g_c)$ amounts to solving the functional equation for the coefficients $C_{\vec[k],\vec[m]}^{\vec[n]}$ for 
$$v_{\vec[k]}=\sum C_{\vec[k],\vec[m]}^{\vec[n]} x_{\vec[m]}\ox y_{\vec[n]}.$$
\begin{Con} The coefficients $C_{\vec[k],\vec[m]}^{\vec[n]}$ can be represented explicitly by linear combinations of $q$-binomials.
\end{Con}
This seems to be straightforward since the functional equations are just finite difference equations expressed in terms of $q$-numbers. In particular it seems only necessary to calculate the expressions for the highest weight vectors, since they uniquely determined the rest of the basis. General methods to obtain the Clebsch-Gordan coefficients $C_{\vec[k],\vec[m]}^{\vec[n]}$ have been studied for the classical \cite{AKHD} and quantum groups \cite{Ma} by solving similar functional equations. However, we are not aware of any explicit general formula in terms of $q$-binomials and the weight parameters $\vec[N]$ in the quantum level except for the case of $\cU_q(\sl_2)$ and at most the octet representations of $\cU_q(\sl_3)$.

Applying the philosophy of Remark \ref{replaceqd} by normalizing with the theory of canonical basis, and replacing $q$-binomials with quantum dilogarithms, the resulting \emph{integral transformation} is indeed the required intertwiners for the tensor product decomposition of the positive representations $\cP_{\vec[\l_1]}\ox \cP_{\vec[\l_2]}$. Here one has to take into account the summation range of the Clebsch-Gordan equation which becomes nontrivial in higher rank, governed by the Littlewood-Richardson rule (see e.g. \cite{HK}). Also one has to show that the integral transformation is well-defined and unitary under some measure. In particular it should intertwine the action of the \emph{positive Casimirs} calculated in \cite{Ip5}. Investigating small cases of the Littlewood-Richardson rule, we made the following conjecture in \cite{Ip6}:
\begin{Con} The positive representations for $\cU_{q\til{q}}(\g_\R)$ is closed under taking the tensor product, and it decomposes as
\Eq{\label{conj}
\cP_{\vec[\l_1]}\ox \cP_{\vec[\l_2]}\simeq \int_{\R_+^{l(w_0)}}^{\o+} \cP_{\vec[\c]} d\mu(\vec[\c]),
}
where $\vec[\c]=\sum_{\a\in \D_+} \c_\a\w_\a$ summing over all the positive roots $\D_+$, where $\c_\a\in \R_+$ and $\w_\a$ are the fundamental weights, with the abuse of notation $\w_\a:=\w_{\a_1}+\w_{\a_2}$ if $\a:=\a_1+\a_2$ is not simple. The Plancherel measure $d\mu(\vec[\c])$ is a continuous measure given by 
\Eq{
d\mu(\vec[\c])=\prod_{\a\in\D_+}\sinh(2\pi b\c_\a)\sinh(2\pi b\inv \c_\a)d\c_\a.} 
Note that both sides of \eqref{conj} is isomorphic to $L^2(\R^{2l(w_0)})$.
\end{Con}
We believe that even establishing the conjecture for lower rank case of $\cU_{q\til{q}}(\sl(3,\R))$ is enough to provide major breakthroughs in the theory of positive representations of split real quantum groups and its many applications as a completely new class of braided tensor categories.


\end{document}